\documentclass[12pt]{amsart}

\usepackage{hyperref}
\hypersetup{nesting=true,debug=true,naturalnames=true}
\usepackage{graphicx,amssymb,upref}

\usepackage{amsmath,amsthm}
\usepackage{amsfonts}
\usepackage{latexsym}
\usepackage{amsbsy}
\usepackage{amssymb,mathscinet}
\usepackage{tikz-cd}
\numberwithin{equation}{section}

\usepackage{amsfonts,amssymb,amscd,amsmath,enumerate,verbatim,calc,enumerate}
\theoremstyle{plain}
\newtheorem{theorem}{Theorem}[section]
\newtheorem{corollary}[theorem]{Corollary}
\newtheorem{lemma}[theorem]{Lemma}
\newtheorem{proposition}[theorem]{Proposition}

\newtheorem{definition}[theorem]{Definition}

\newtheorem{remark}[theorem]{Remark}
\newtheorem{example}[theorem]{Example}

\begin{document}

\title[Powers and roots of partial isometric covariant representations]{Powers and roots of partial isometric covariant representations}

\date{\today}
\author[Saini]{Dimple Saini\textsuperscript{*}}
\address{Centre for mathematical and Financial Computing, Department of Mathematics, The LNM Institute of Information Technology, Rupa ki Nangal, Post-Sumel, Via-Jamdoli Jaipur-302031, (Rajasthan) India}
\email{18pmt006@lnmiit.ac.in, dimple92.saini@gmail.com}
\author[Trivedi]{Harsh Trivedi}
\address{Centre for mathematical and Financial Computing, Department of Mathematics, The LNM Institute of Information Technology, Rupa ki Nangal, Post-Sumel, Via-Jamdoli Jaipur-302031, (Rajasthan) India}
\email{harsh.trivedi@lnmiit.ac.in, trivediharsh26@gmail.com}
\author[Veerabathiran]{Shankar Veerabathiran}
\address{Institute of Mathematical Sciences,
	A CI of Homi Bhabha National Institute,
	CIT Campus, Taramani, Chennai, 600113,
	Tamilnadu, INDIA.}
\email{shankarunom@gmail.com}

%\tableofcontents

\begin{abstract}
	Isometric covariant representations play an important role in the study of Cuntz-Pimsner algebras. In this article, we study partial isometric covariant representations and explore under what conditions powers and roots of partial isometric covariant representations are also partial isometric covariant representations.
\end{abstract}


\maketitle

\section{Introduction}
The classical result of Wold \cite{W38} asserts that: A given isometry on a Hilbert space is either a shift, or a unitary, or breaks uniquely as a direct summand of both. Let us denote by $B(\mathcal{H})$ the algebra of all bounded linear operators on a Hilbert space $\mathcal{H},$ and for $V\in B(\mathcal{H})$ we denote  the kernel of $V$ and the range of $V$ by $N(V)$ and $R(V)$, respectively.  A bounded linear operator $V$ on a Hilbert space is said to be {\it partial isometry} if it is isometry on the orthogonal complement of $N({V}).$  It is clear that $V$ is a partial isometry if and only if $VV^*$ is a (final) projection onto $R(V)$ if and only if $V^*V$ is a (initial) projection onto $R(V^*).$ It is well known that the sum and multiplication of partial isometries need not be partial isometries. Murray and von Neumann \cite{MN36} discussed the classification of factors using the sum of partial isometries. In the following result Erdelyi \cite{E68} studied when the product of two partial isometries is also a partial isometry: 
\begin{theorem}
	Suppose that $V_1$ and $V_2$ are two partial isometries acting on a Hilbert space $\mathcal{H}.$ Let $T=V_1V_2.$ Then the following conditions are equivalent:
\begin{itemize}
\item[(1)] $T$ is a partial isometry.
\item[(2)] Initial space of $V_1,$ $ R(V_1^*),$ is invariant under $P_{R(V_2)},$ that is,
$$V_2V_2^* R(V_1^*)\subseteq R(V_1^*).$$
\item[(3)] $R(V_2)$ is invariant under $P_{R(V_1^*)},$ that is,
$$V_1^*V_1 R(V_2)\subseteq R(V_2).$$
\item[(4)] The product $P_{R(V_1^*)}P_{R(V_2)}$ is idempotent operator.
\end{itemize}
\end{theorem}

 A bounded linear operator $V$ on a Hilbert space is called {\it power partial isometry} if $V^n$ is a partial isometry for all $n\ge 1.$
Halmos and Wallen \cite{HW69} explored power partial isometries and proved that every power partial isometry uniquely decomposes into a direct sum of a pure isometry, a unitary operator, a pure co-isometry and a truncated shift.  Ezzahraoui, Mbekhta, Salhi and Zerouali \cite{EMSZ18} characterized the powers of a partial isometry in terms of orthogonal complement of kernels. 	For $n\in \mathbb{N},$ we denote $I_n:=\{1,2,\dots,n\}.$

\begin{theorem}
	Suppose that $V$ is a partial isometry acting on a Hilbert space $\mathcal{H}.$ Then the following conditions are equivalent:
	\begin{itemize}
		\item[(1)] $V^k$ is a partial isometry for all $k\in I_n.$
		\item[(2)] $V(N(V^k)^{\perp})\subseteq N(V^{k-1})^{\perp}$ for all $k\in I_n.$
	\end{itemize} Moreover, if $V$ is a power partial isometry if and only if $V(N(V^{k+1})^{\perp})\subseteq N(V^{k})^{\perp}$ for all $k\ge 1.$
\end{theorem} Gupta \cite{BCG80} analyzed the converse question and proved that 
if $V\in B(\mathcal{H})$ is a contraction with $N(V^2)=N(V),$ and if $V^n$ is a partial isometry for some $n\ge 2,$ then $V$ is a partial isometry.  Ezzahraoui, Mbekhta, Salhi and Zerouali \cite{EMSZ18} generalized the Gupta's result in the following way:
\begin{theorem}
	Suppose that $V$ is a contraction on a Hilbert space $\mathcal{H}.$ Let $V^n$ is a partial isometry for some $n\ge 2.$ Then the following conditions are equivalent:
	\begin{itemize}
		\item[(1)] $V$ is a partial isometry.
		\item[(2)] \begin{itemize}
			\item[(a)] For every $h\in N(V^n)\ominus N(V)$ such that $\|Vh\|=\|h\|.$
			\item[(b)] $V(N(V^n)^{\perp})$ is orthogonal to $V( N(V^n)\ominus N(V)).$
		\end{itemize}
	\end{itemize}
\end{theorem}
In a noncommutative context, Cuntz \cite{C77} investigated a $C^*$-algebra, famously known as the Cuntz algebra, which is generated by row isometry. The tuple $(V_1,\dots,V_n)$ of isometries on a Hilbert space $\mathcal{ H}$ is said to be {\it row isometry} if $\sum_{i=1}^{n} V_iV_i^*=I_{\mathcal{ H}},$ or equivalently, orthogonal range. Frazho introduced the Wold decomposition for two isometries with orthogonal range, and Popescu extended this decomposition to encompass the scenario of an infinite sequence of isometries with orthogonal final spaces.  Pimsner in \cite{P97} extended the construction of Cuntz-Kreiger algebras using isometric covariant representations of $C^*$-correspondences. Therefore the setup of isometric  covariant representations provides a general framework to study row isometries using the tensor product notion of $C^*$-correspondences. A celebrated theorem due to Pimsner \cite{P97} states that a covariant representation $V$ of a $C^*$-correspondence $E$ extends to a $C^*$-representation of the Toeplitz algebra
	of $E$ if and only if $V$ is isometric. 
	
	 Muhly and Solel derived an analogue of, Popescu's Wold decomposition \cite{POP89} for a row isometry known as the Wold decomposition, for isometric covariant representations in \cite{MS99} which is based on Fock module approach. In \cite{Ami11}, Viselter studied powers of partial isometric covariant representations and introduced relatively isometric covariant representations to explore conditions for a covariant representation of a subproduct system to extend to a $C^*$-representation of the Toeplitz algebra. Our aim in this article is to study conditions so that powers and roots of partial isometric covariant representations are also partial isometric covariant representations.
	
	The section-wise plan is as follows: In Section 1, we recall the basic setup and definitions. In Section 2, we study hypothesis under which the multiplication of partial isometric representations of $C^*$-correspondences becomes partial isometric representations. In Section 3, we analyze under which assumptions the powers and roots of partial isometric representations become partial isometric representations.

\subsection{Preliminaries and Notations}

Here, we will review some basic concepts from \cite{L95, MS98, MS99, P97}.	
Let $E$ be a Hilbert $C^*$-module over a $C^*$-algebra $\mathcal{B}.$ It is called $C^*$-correspondence over $\mathcal{B}$ if $E$ has a left $\mathcal{B}$-module action by a non-zero $*$-homomorphism $\phi:\mathcal B\to \mathcal L(E),$ that is, $ b\xi :=\phi(b)\xi$ for all $b\in\mathcal B$ and $\xi\in E,$ where $\mathcal L(E)$ is the collection of all adjointable operators on $E.$ If the closed linear span of $\{\langle z,w\rangle : z,w\in E \}$ equals $\mathcal{B},$ then we say that $E$ is {\it full}. In this paper, each $*$-homomorphism considered is nondegenerate, that is, the closed linear span of $\phi(\mathcal{B})E$ equals $E.$ The following definition of completely bounded covariant representations plays an important role in this paper:

\begin{definition}
	Let map $V:E\to B(\mathcal H)$ be linear and $\sigma:\mathcal B\to B(\mathcal H)$ be a representation. Then the pair $(\sigma,V)$ is
	said to be a {\rm  covariant representation} (cf. \cite{MS98}) of $E$ on $\mathcal H$  if 
	\[
	V(b\xi c)=\sigma(b)V(\xi)\sigma(c) \quad \quad (\xi\in E,
	b,c\in\mathcal B).
	\] 
	Then $(\sigma,V)$ is called {\rm completely bounded covariant representation} (simply say, {\rm c.b.c. representation})  if $V$ is completely bounded. Further, $(\sigma, V)$ is called {\rm isometric} if $V(\xi)^*V(\eta)=\sigma(\langle \xi,\eta \rangle)$ for all $\xi,\eta\in E.$
\end{definition}

The following lemma is due to Muhly and Solel \cite{MS98} which is important to classify the c.b.c. representations of a $C^*$-correspondence:

\begin{lemma} \label{MSC}
	The map $(\sigma, V)\mapsto \widetilde{V}$ provides a bijection between the set of all c.b.c. (respectively, completely contractive covariant) representations $(\sigma, V)$ of $E$ on $\mathcal H$ and the set of all bounded (respectively,
	contractive) linear maps $\widetilde{V}:E\otimes_{\sigma} \mathcal H\to \mathcal
	H$ defined by
	\[
	\widetilde{V}(\eta\otimes h)=V(\eta)h \quad \quad (
	h\in\mathcal H,\eta\in E),
	\] such that $\sigma(b)\widetilde{V}=\widetilde{V}(\phi(b)\otimes I_{\mathcal
		H})$ for all $b\in\mathcal B$. Moreover, $\widetilde{V}$ is {\rm isometry} if and only if $(\sigma, V)$ is isometric.
\end{lemma}

For every $n\in \mathbb{N}_0(=\mathbb{N}\cup \{0\})$,
$E^{\otimes n} =E\otimes_{\phi} \dots \otimes_{\phi}E$ ($n$-times) (here $E^{\otimes 0} =\mathcal{B}$) is a $C^*$-correspondence over $\mathcal{B}$, with the left module action  of $\mathcal{B}$ on $E^{\otimes n}$ defined as  $$\phi_n(b)(\xi_1 \otimes \dots \otimes \xi_n)=b\xi_1\otimes \dots \otimes\xi_n, \: \:\:\: b\in \mathcal{B},\xi_i \in E.$$
For  $n\in \mathbb{N},$ define $\widetilde{V}_n : E^{\otimes n}\otimes \mathcal{H} \to \mathcal{H}$ by $$\widetilde{V}_n (\xi_1 \otimes \dots \otimes \xi_n \otimes h) = V (\xi_1) \dots V(\xi_n) h, \quad \xi_i \in E, h \in \mathcal H.$$ 
The \emph{Fock module} of $E$ (cf. \cite{F02}), $\mathcal{F}(E)= \bigoplus_{n \geq 0}E^{\otimes n},$ is a $C^*$-correspondence over $\mathcal{B},$ where the left module action  of $\mathcal{B}$ on $\mathcal{F}(E)$ is defined  by $$\phi_{\infty}(b)\left(\oplus_{n \geq 0}\xi_n\right)=\oplus_{n \geq 0}\phi_n(b)\xi_n , \:\: \xi_n \in E^{\otimes n}.$$
For $\xi \in E,$  the \emph{creation operator} $V_{\xi}$ on $\mathcal{F}(E)$ is defined by $$V_{\xi}(\eta)=\xi \otimes \eta, \:\: \eta \in E^{\otimes n}, n\ge 0.$$ 

The following theorem is due to Pimsner \cite{P97}.

\begin{theorem} 
	Let $\mathcal{T}{(E)}$ be a $C^*$-algebra in $\mathcal{L}(\mathcal{F}(E)),$ called {\rm Toeplitz algebra}, generated by $\{V_{\xi}\}_{\xi\in E}$ and $\{\phi_{\infty}(b)\}_{b\in \mathcal{B}}.$ 
	Let $(\sigma,V)$ be an isometric covariant representation of $E$ on $\mathcal H.$ Then the map $$\begin{cases}
		V_{\xi}\mapsto V(\xi), &   \xi\in E; \\
		\phi_{\infty}(b)\mapsto \sigma(b), &   b\in \mathcal{B}
	\end{cases}$$ extends uniquely to a $C^*$-representation of $\mathcal{T}(E)$ on $\mathcal{H}.$ Conversely, suppose that $\pi: \mathcal{T}(E)\mapsto B(\mathcal{H})$ is a $C^*$-representation, and if $V(\xi)$ is defined to be $\pi(V_{\xi})$ for $\xi\in E,$ and $\sigma(b)$ is defined to be $\pi(\phi_{\infty}(b))$ for $b\in \mathcal{B},$ then $(\sigma,V)$ is an isometric covariant representation of $E$ on $\mathcal{H}.$
\end{theorem}

The following notion of partial isometric covariant representation plays an important role in this paper:

\begin{definition}	
Let $(\sigma,V)$  be a  c.b.c. representation of ${E}$ on $\mathcal{H}$. We say that $(\sigma,V)$ is {\rm partial isometric} if $\|\widetilde{V}x\|=\|x\|$ for all $x\in N(\widetilde{V})^{\perp}$. 
\end{definition}

Let $(\sigma,V)$ be a c.b.c. representation of $E$ on $\mathcal{H}$ with closed range. The {\it Moore-Penrose inverse} $\widetilde{V}^{\dagger}$ of $(\sigma,V)$ \cite{AHS22} is defined as the unique solution of the following four equations: \begin{equation*}
	\widetilde{V}\widetilde{V}^{\dagger}\widetilde{V}=\widetilde{V},\quad \widetilde{V}^{\dagger}\widetilde{V}\widetilde{V}^{\dagger}=\widetilde{V}^{\dagger},\quad (\widetilde{V}\widetilde{V}^{\dagger})^*=\widetilde{V}\widetilde{V}^{\dagger},\quad( \widetilde{V}^{\dagger}\widetilde{V})^*=\widetilde{V}^{\dagger}\widetilde{V}.
\end{equation*}
For $n\in \mathbb{N},$ define $\widetilde{V}^{\dagger (n)}:\mathcal{H}\to E^{\otimes n}\otimes \mathcal{H}$ by $$\widetilde{V}^{\dagger(n)}:= (I_{E^{\otimes n-1}}\otimes \widetilde{V}^{\dagger})(I_{E^{\otimes n-2}}\otimes\widetilde{V}^{\dagger})\dots(I_{E}\otimes \widetilde{V}^{\dagger})\widetilde{V}^{\dagger}.$$

\begin{remark}\cite{K11}
	Let $(\sigma,V)$ be a c.b.c. representation of $E$ on $\mathcal{H}$. Then the following conditions are equivalent:
	\begin{enumerate}
		\item  $\widetilde{V}$ is a partial isometry.   $\:\:\:\:\:\:\:\:\:\:\:\:\:\:\:\:\:\:$   $(4)$ $\widetilde{V}^*\widetilde{V}=P_{{R}(\widetilde{V}^*)}$ (initial
		projection of $\widetilde{V}$).
		\item $\widetilde{V}^*$ is a partial isometry.   $\:\:\:\:\:\:\:\:\:\:\:\:\:\:\:\:$   $(5)$ $\widetilde{V}\widetilde{V}^*=P_{{R}(\widetilde{V})}$  (final
		projection of $\widetilde{V}$).
		\item   $\widetilde{V}\widetilde{V}^*\widetilde{V}=\widetilde{V}.$    $\:\:\:\:\:\:\:\:\:\:\:\:\:\:\:\:\:\:\:\:\:\:\:\:\:\:\:\:\:\:\:\:\:\:\:\:\:\:\:\:\:$    $(6)$ $\widetilde{V}^{\dagger}=\widetilde{V}^*.$  
		\end{enumerate}
\end{remark}

\begin{remark}
	Let $(\sigma,V)$ be a partial isometric representation of $E$ on $\mathcal{H}.$ If $\mathcal{K}$ is a $(\sigma,V)$-reducing subspace of $\mathcal{H},$ then $(\sigma,V)|_{\mathcal{K}}$ is also a partial isometric representation (see \cite[Lemma 1]{HW69}) of $E$ on $\mathcal{K}.$
		
\end{remark}
%\begin{proof}
%	Let $\mathcal{K}$ be a $(\sigma,V)$-reducing subspace, then we define a c.b.c representation $(\psi,T):=(\sigma,V)|_{\mathcal{K}},$ where $\psi=\sigma|_{\mathcal{K}}$ and $T(\xi)=V(\xi)|_{\mathcal{K}}.$ It gives $$\widetilde{T}\widetilde{T}^*\widetilde{T}=\widetilde{V}\widetilde{V}\widetilde{V}|_{\mathcal{K}}=\widetilde{V}|_{\mathcal{K}}=\widetilde{T}.$$ Therefore $(\sigma,V)|_{\mathcal{K}}$ is a partial isometric representation.
	%Since $(\sigma,V)$ is a partial isometric c.b.c. representation and $\mathcal{K}$ is a $(\sigma,V)$-reducing subspace of $\mathcal{H}.$ Suppose that $P_{\mathcal{K}}$ is an orthogonal projection of $\mathcal{H}$ onto $\mathcal{K}$ such that $\widetilde{V}(I_E \otimes P_{\mathcal{K}})=P_{\mathcal{K}}\widetilde{V}$. 
	%If $\mathfrak{E}$ = $\widetilde{V}^*\widetilde{V}$ is the initial projection, then $$\mathfrak{E}(I_E \otimes P_{\mathcal{K}})=(I_E \otimes P_{\mathcal{K}})\mathfrak{E}.$$
	%Since $(I_E \otimes P_{\mathcal{K}})$ = $(I_E \otimes P_{\mathcal{K}})\mathfrak{E}+(I_E \otimes P_{\mathcal{K}})(I_{E\otimes\mathcal{H}}-{\mathfrak{E}})$, it gives $R(I_E \otimes P_{\mathcal{K}})$ is the direct sum of the orthogonal subspaces $R(I_E \otimes P_{\mathcal{K}})\cap N(\widetilde{V})^\perp$ and $R(I_E\otimes P_{\mathcal{K}})\cap N(\widetilde{V})$. The lemma follows by the definition of partial isometric c.b.c. representation.
%\end{proof}

%It is well known that if $(\sigma,V)$ is a partial isometric representation, then $\widetilde{V}_n$ need not be partial isometry for $n>1.$ 

\begin{definition}
	A c.b.c. representation $(\sigma,V)$ of $E$ on $\mathcal{H}$ is called {\rm power partial isometric} representation if each $\widetilde{V}_n$ is a partial isometry for all $n\in \mathbb{N}.$
\end{definition}

\section{Product of Partial isometric covariant representations}
 In this section, we will discuss the product of partial isometric representations of $C^*$-correspondences.  Let $(\sigma , V^{(1)})$ and $(\sigma, V^{(2)})$ be two partial isometric representations of $C^*$-correspondences $E_1$ and $E_2$ on $\mathcal{H},$ respectively. Define a c.b.c. representation $(\sigma,V)$ of $E_1\otimes E_2$ on $\mathcal{H}$ by $\widetilde{V}:=\widetilde{V}^{(1)}(I_{E_1} \otimes \widetilde{V}^{(2)}).$ If $\widetilde{V}^{(1)}(I_{E_1}\otimes \widetilde{V}^{(2)}\widetilde{V}^{(2)^*})=\widetilde{V}^{(2)}\widetilde{V}^{(2)^*}\widetilde{V}^{(1)},$ then $(\sigma,V)$ is a partial isometric representation. Indeed,
 \begin{align*}
 	\widetilde{V}\widetilde{V}^*\widetilde{V}&=\widetilde{V}^{(1)}(I_{E_1} \otimes \widetilde{V}^{(2)}\widetilde{V}^{(2)^*})\widetilde{V}^{(1)^*}\widetilde{V}^{(1)}(I_{E_1} \otimes \widetilde{V}^{(2)})=\widetilde{V}^{(2)}\widetilde{V}^{(2)^*}\widetilde{V}^{(1)}\widetilde{V}^{(1)^*}\widetilde{V}^{(1)}(I_{E_1} \otimes \widetilde{V}^{(2)})\\&=\widetilde{V}^{(2)}\widetilde{V}^{(2)^*}\widetilde{V}^{(1)}(I_{E_1} \otimes \widetilde{V}^{(2)})=\widetilde{V}^{(1)}(I_{E_1}\otimes \widetilde{V}^{(2)}\widetilde{V}^{(2)^*}\widetilde{V}^{(2)})=\widetilde{V}.
 \end{align*}

 The following result is a generalization of \cite[Lemma 2]{HW69}.
 \begin{theorem}
 	Let $(\sigma , V^{(1)})$ and $(\sigma, V^{(2)})$ be two partial isometric representations of $C^*$-correspondences $E_1$ and $E_2$ on $\mathcal{H},$ respectively. Then the following conditions are equivalent:
 	\begin{enumerate}
 		\item  $\widetilde{V}$ = $\widetilde{V}^{(1)}(I_{E_1} \otimes \widetilde{V}^{(2)})$ is a partial isometry.
 		\item  The projections $\mathfrak{E}:=\widetilde{V}^{{(1)}^*}\widetilde{V}^{(1)}$ and $\mathfrak{F}:=(I_{E_1} \otimes{\widetilde{V}^{(2)}}\widetilde{V}^{{(2)}^*})$ commute.
 	\end{enumerate}
 \end{theorem}
 \begin{proof}
 	Since $\mathfrak{E}$ and $\mathfrak{F}$ commute, we get
 	\begin{align*}
 		\widetilde{V}\widetilde{V}^*\widetilde{V} &= \widetilde{V}^{{(1)}}(I_{E_1} \otimes \widetilde{V}^{{(2)}}\widetilde{V}^{{(2)}^*}) \widetilde{V}^{{(1)}^*}\widetilde{V}^{{(1)}}(I_{E_1} \otimes \widetilde{V}^{{(2)}}) =\widetilde{V}^{{(1)}}\mathfrak{F}\mathfrak{E}(I_{E_1} \otimes \widetilde{V}^{{(2)}}) \\&= \widetilde{V}^{{(1)}}\mathfrak{E}\mathfrak{F}(I_{E_1} \otimes \widetilde{V}^{{(2)}}) = \widetilde{V}^{(1)}(I_{E_1} \otimes \widetilde{V}^{(2)})= \widetilde{V}.
 	\end{align*}
 	Conversely, suppose that $\widetilde{V}$ = $\widetilde{V}\widetilde{V}^*\widetilde{V},$ then 
 	$\widetilde{V}^{{(1)}}\mathfrak{F}\mathfrak{E}(I_{E_1} \otimes \widetilde{V}^{{(2)}})= \widetilde{V}^{(1)}(I_{E_1} \otimes \widetilde{V}^{(2)}),$
 	and hence
 	\begin{align*}
 		\mathfrak{E}\mathfrak{F}\mathfrak{E}\mathfrak{F} = \widetilde{V}^{{(1)}^*}\widetilde{V}^{{(1)}}\mathfrak{F}\mathfrak{E}(I_{E_1} \otimes \widetilde{V}^{{(2)}}\widetilde{V}^{{(2)}^*})=\widetilde{V}^{{(1)}^*}\widetilde{V}^{(1)}(I_{E_1} \otimes \widetilde{V}^{(2)}\widetilde{V}^{{(2)}^*})=\mathfrak{E}\mathfrak{F}.
 	\end{align*}
 	Therefore $\mathfrak{E}\mathfrak{F}$ is idempotent. Since $\|\mathfrak{E}\mathfrak{F}\|$ $\le 1,$ we have $\mathfrak{E}\mathfrak{F}$ is self-adjoint, and thus $\mathfrak{E}$ and $\mathfrak{F}$ commute.
 \end{proof}

 The following result is regarding the product of a finite number of partial isometric representations which is a generalization of \cite[Theorem 2]{E68}:

\begin{theorem}\label{www01}
	Let $E_1,\dots, E_n$ be $C^*$-correspondences. For each $i\in I_n,$ let $(\sigma , V^{(i)})$ be a partial isometric representation of $E_i$ on $\mathcal{H}.$ For $i\in I_n,$ we define a c.b.c. representation $(\sigma,T^{(i)})$ of a $C^*$-correspondence $E_1\otimes \dots \otimes E_i$ on $\mathcal{H}$ by $$T^{(i)}(\xi_1\otimes\xi_2\otimes \dots\otimes\xi_i)h=V^{(1)}(\xi_1)V^{(2)}(\xi_2)\dots V^{(i)}(\xi_i)h\quad \quad (\xi_1\in E_1,\dots, \xi_i\in E_i,h\in \mathcal{H} ).$$
	Then the following conditions are equivalent:
	\begin{enumerate}
		\item  $(\sigma,T^{(2)}),\dots,(\sigma,T^{(n)})$  are  partial isometric representations.
		\item Range of $\widetilde{T}^{(i-1)^*}$ is invariant under $P_{{R}(I_{E_1\otimes \dots \otimes E_{i-1}}\otimes\widetilde{V}^{(i)})}$, that is, $$(I_{E_1\otimes \dots \otimes E_{i-1}}\otimes\widetilde{V}^{(i)}\widetilde{V}^{(i)^*}){R}(\widetilde{T}^{(i-1)^*}) \subseteq {R}(\widetilde{T}^{(i-1)^*})\quad for \quad all \quad i\in I_{n}\setminus \{1\}. $$
		\item Range of $(I_{E_1\otimes \dots \otimes E_i}\otimes\widetilde{V}^{(i+1)})$ is invariant under $P_{{R}(\widetilde{T}^{(i)^*})}$, that is, $$\widetilde{T}^{(i)^*}\widetilde{T}^{(i)}{R}(I_{E_1\otimes \dots \otimes E_i}\otimes\widetilde{V}^{(i+1)}) \subseteq {R}(I_{E_1\otimes \dots \otimes E_i}\otimes\widetilde{V}^{(i+1)})\quad for \quad all \quad i\in I_{n-1}.$$
		\item The operator products $P_{{R}(\widetilde{T}^{(i)^*})}P_{{R}(I_{E_1\otimes \dots \otimes E_i}\otimes\widetilde{V}^{(i+1)})} , i\in I_{n-1}$ are idempotent.
	\end{enumerate}
 \end{theorem}
\begin{proof}
 We will prove it using the mathematical induction. For $n = 2$ we have 
 
 	$(1)\Rightarrow (2)$ We will prove it using contradiction. Suppose $(2)$ does not hold, then there exists a nonzero element $\eta \in {R}(\widetilde{T}^{(2)^*})=(I_{E_1} \otimes\widetilde{V}^{(2)^*})R(\widetilde{V}^{(1)^*})$ such that $\zeta=(I_{E_1} \otimes\widetilde{V}^{(2)})\eta \notin {R}(\widetilde{V}^{(1)^*}).$
 Note that $\langle \zeta,\zeta\rangle$= $\langle(I_{E_1} \otimes\widetilde{V}^{(2)})\eta,(I_{E_1} \otimes\widetilde{V}^{(2)})\eta\rangle$=$\langle \eta,\eta\rangle.$ Since $\widetilde{V}_2$ is a partial isometry we have	$$\langle \eta,\eta\rangle=\langle \widetilde{T}^{(2)}\eta,\widetilde{T}^{(2)}\eta\rangle=\langle \widetilde{V}^{(1)}\zeta,\widetilde{V}^{(1)}\zeta\rangle < \langle \zeta,\zeta\rangle=\langle \eta,\eta\rangle.$$
 This is a contradiction, so $(I_{E_1} \otimes\widetilde{V}^{(2)}\widetilde{V}^{(2)^*}){R}(\widetilde{V}^{(1)^*}) \subseteq {R}(\widetilde{V}^{(1)^*}).$
 
 $(2)\Rightarrow (1)$ Suppose $(2)$ holds, then $\widetilde{V}^{(1)^*}\widetilde{V}^{(1)}$ is the orthogonal projection onto $(I_{E_1}\otimes\widetilde{V}^{(2)}){R}(\widetilde{T}^{(2)^*}).$ Let  $\eta \in {R}(\widetilde{T}^{(2)^*})$ $\subseteq$ ${R}(I_{E_1}\otimes\widetilde{V}^{(2)^*})$ we have  $$\langle\widetilde{T}^{(2)}\eta,\widetilde{T}^{(2)}\eta\rangle=\langle \widetilde{V}^{(1)}(I_{E_1} \otimes\widetilde{V}^{(2)})\eta,\widetilde{V}^{(1)}(I_{E_1} \otimes\widetilde{V}^{(2)})\eta\rangle=\langle (I_{E_1} \otimes\widetilde{V}^{(2)})\eta,(I_{E_1} \otimes\widetilde{V}^{(2)})\eta\rangle=\langle \eta,\eta\rangle.$$ 
 
 $(1)\Rightarrow (3)$ Suppose that $\widetilde{T}^{(2)}$ is a partial isometry, then $\widetilde{T}^{(2)^*}$ is also a partial isometry. Apply condition $(2)$ to the adjoint $\widetilde{T}^{(2)^*}=(I_{E_1}\otimes\widetilde{V}^{(2)^*})\widetilde{V}^{(1)^*},$ we get $(3).$

 $(3)\Rightarrow (1)$ Suppose $(3)$ holds, then $(I_{E_1}\otimes\widetilde{V}^{(2)}\widetilde{V}^{(2)^*})$ is the orthogonal projection onto $\widetilde{V}^{(1)^*}{R}(\widetilde{T}^{(2)})$. Let $h\in {R}(\widetilde{T}^{(2)}) \subseteq {R}(\widetilde{V}^{(1)})$ we get  $$\langle \widetilde{T}^{(2)^*}h,\widetilde{T}^{(2)^*}h\rangle=\langle(I_{E_1}\otimes\widetilde{V}^{(2)^*})\widetilde{V}^{(1)^*}h,(I_{E_1}\otimes\widetilde{V}^{(2)^*})\widetilde{V}^{(1)^*}h\rangle=\langle\widetilde{V}^{(1)^*}h,\widetilde{V}^{(1)^*}h\rangle=\langle h,h\rangle.$$ Therefore $\widetilde{T}^{(2)^*}$ is a partial isometry, and hence $\widetilde{T}^{(2)}$ is a partial isometry.

 $(1)\Rightarrow (4)$ Suppose that $\widetilde{T}^{(2)}$ is a partial isometry, then $\widetilde{T}^{(2)}=\widetilde{T}^{(2)}\widetilde{T}^{(2)^*}\widetilde{T}^{(2)},$ that is, $$\widetilde{V}^{(1)}(I_{E_1}\otimes\widetilde{V}^{(2)})=\widetilde{V}^{(1)}(I_{E_1} \otimes \widetilde{V}^{(2)}\widetilde{V}^{(2)^*})\widetilde{V}^{(1)^*}\widetilde{V}^{(1)}(I_{E_1} \otimes \widetilde{V}^{(2)}).$$ Pre-multiply by $\widetilde{V}^{(1)^*}$ and post-multiply by $(I_{E_1}\otimes\widetilde{V}^{(2)^*})$ in the last inequality, we get $$\widetilde{V}^{(1)^*}\widetilde{V}^{(1)}(I_{E_1}\otimes\widetilde{V}^{(2)}\widetilde{V}^{(2)^*})=\widetilde{V}^{(1)^*}\widetilde{V}^{(1)}(I_{E_1} \otimes \widetilde{V}^{(2)}\widetilde{V}^{(2)^*})\widetilde{V}^{(1)^*}\widetilde{V}^{(1)}(I_{E_1} \otimes \widetilde{V}^{(2)}\widetilde{V}^{(2)^*}).$$ This implies that $P_{{R}(\widetilde{V}^{(1)^*})}P_{{R}(I_{E_1}\otimes \widetilde{V}^{(2)})}$ is idempotent.
 
 $(4)\Rightarrow (1)$ Since  $P_{R(\widetilde{V}_1^*)}P_{R(I_{E_1}\otimes \widetilde{V}_2)}$ is idempotent, we have 
 \begin{align*}
 \widetilde{T}^{(2)}\widetilde{T}^{(2)^*}\widetilde{T}^{(2)}&=\widetilde{V}^{(1)}(I_{E_1} \otimes \widetilde{V}^{(2)}\widetilde{V}^{(2)^*})\widetilde{V}^{(1)^*}\widetilde{V}^{(1)}(I_{E_1} \otimes \widetilde{V}^{(2)}) \\&=\widetilde{V}^{(1)}\widetilde{V}^{(1)^*}\widetilde{V}^{(1)}(I_{E_1} \otimes \widetilde{V}^{(2)}\widetilde{V}^{(2)^*})\widetilde{V}^{(1)^*}\widetilde{V}^{(1)}(I_{E_1} \otimes \widetilde{V}^{(2)}\widetilde{V}^{(2)^*}\widetilde{V}^{(2)})\\&=\widetilde{V}^{(1)}\widetilde{V}^{(1)^*}\widetilde{V}^{(1)}(I_{E_1} \otimes \widetilde{V}^{(2)}\widetilde{V}^{(2)^*}\widetilde{V}^{(2)})=\widetilde{V}^{(1)}(I_{E_1}\otimes \widetilde{V}^{(2)})= \widetilde{T}^{(2)}.
 \end{align*}
 Therefore $ \widetilde{T}^{(2)}$ is a partial isometry.
 Suppose that $(1), (2),(3)$ and $(4)$ are equivalent for $n = m.$ Consider
	$$ \widetilde{T}^{(m+1)}= \widetilde{T}^{(m)}(I_{E_1\otimes \dots \otimes E_m}\otimes\widetilde{V}^{(m+1)}),$$ where $\widetilde{V}^{(m+1)}$ is a partial isometry. Then $\widetilde{T}^{(m+1)}$ is a partial isometry  if and only if any of the following conditions hold:
	\begin{itemize}
		\item[(1)] $(I_{E_1\otimes \dots \otimes E_m}\otimes\widetilde{V}^{(m+1)}\widetilde{V}^{(m+1)^*}){R}(\widetilde{T}^{(m)^*}) \subseteq {R}(\widetilde{T}^{(m)^*}).$\\
		
		\item[(2)] $\widetilde{T}^{(m)^*}\widetilde{T}^{(m)}{R}(I_{E_1\otimes \dots \otimes E_m}\otimes\widetilde{V}^{(m+1)}) \subseteq {R}(I_{E_1\otimes \dots \otimes E_m}\otimes\widetilde{V}^{(m+1)}).$\\
		
		\item[(3)] $(P_{{R}(\widetilde{T}^{(m)^*})}P_{{R}(I_{E_1\otimes \dots \otimes E_m}\otimes\widetilde{V}^{(m+1)})})^2=P_{{R}(\widetilde{T}^{(m)^*})}P_{{R}(I_{E_1\otimes \dots \otimes E_m}\otimes\widetilde{V}^{(m+1)})}.$
	\end{itemize}
This completes the proof.
	\end{proof}

\begin{remark}
	Let $E_1,\dots, E_n$ be $C^*$-correspondences. For each $i\in I_n,$ let $(\sigma , V^{(i)})$ be a partial isometric representation of $E_i$ on $\mathcal{H}.$ Define a c.b.c. representation $(\sigma,T^{(n)})$ of a $C^*$-correspondence $E_1\otimes \dots \otimes E_n$ on  $\mathcal{H}$ by $$T^{(n)}(\xi_1\otimes\xi_2\otimes\dots\otimes\xi_n)h:=V^{(1)}(\xi_1)V^{(2)}(\xi_2)\dots V^{(n)}(\xi_n)h\quad \quad (\xi_1\in E_1,\dots, \xi_n\in E_n,h\in \mathcal{H} ).$$ Then ${\widetilde{T}^{(n)}}=\widetilde{V}^{(1)}(I_{E_1}\otimes \widetilde{V}^{(2)})\dots(I_{E_1\otimes \dots \otimes E_{n-1}}\otimes \widetilde{V}^{(n)})$ is a partial isometry if and only if $${\widetilde{T}^{(n)^{\dagger}}}=(I_{E_1\otimes \dots \otimes E_{n-1}}\otimes\widetilde{V}^{(n)^{\dagger}})\dots (I_{E_1}\otimes\widetilde{V}^{(2)^{\dagger}})\widetilde{V}^{(1)^{\dagger}}.$$
\end{remark}
\begin{proof}
	Let ${\widetilde{T}^{(n)}}$ be a partial isometry, then ${\widetilde{T}^{(n)^{\dagger}}}= {\widetilde{T}^{(n)^*}}.$ It follows that
	\begin{align*}
		{\widetilde{T}^{(n)^{\dagger}}}&= {\widetilde{T}^{(n)^*}}=(I_{E_1\otimes \dots \otimes E_{n-1}}\otimes\widetilde{V}^{(n)^*})\dots \widetilde{V}^{(1)^*}=(I_{E_1\otimes \dots \otimes E_{n-1}}\otimes\widetilde{V}^{(n)^{\dagger}})\dots \widetilde{V}^{(1)^{\dagger}}.
	\end{align*}
	
	Conversely, suppose that ${\widetilde{T}^{(n)^{\dagger}}}=(I_{E_1\otimes \dots \otimes E_{n-1}}\otimes\widetilde{V}^{(n)^{\dagger}})\dots (I_{E_1}\otimes\widetilde{V}^{(2)^{\dagger}})\widetilde{V}^{(1)^{\dagger}},$ then  
	\begin{align*}
		{\widetilde{T}^{(n)^*}} &=(I_{E_1\otimes \dots \otimes E_{n-1}}\otimes\widetilde{V}^{(n)^*})\dots \widetilde{V}^{(1)^*} =(I_{E_1\otimes \dots \otimes E_{n-1}}\otimes\widetilde{V}^{(n)^{\dagger}})\dots \widetilde{V}^{(1)^{\dagger}}={\widetilde{T}^{(n)^{\dagger}}}.
	\end{align*}
	Thus ${\widetilde{T}^{(n)}}$ is a partial isometry.
\end{proof}

Let $(\sigma , V^{(1)})$ be a completely contractive covariant representation of $E_1$ on $\mathcal{H}.$ Then the operator matrix $\begin{bmatrix}
	\widetilde{V}^{(1)} & (I_H - \widetilde{V}^{(1)}\widetilde{V}^{(1)^*})^{1/2} \\
	0 & 0 \\
\end{bmatrix}: (E_1 \otimes \mathcal{H}) \oplus \mathcal{H}\to  \mathcal{H} \oplus \mathcal{H}$ is a partial isometry.

\begin{theorem}
	Let $(\sigma , V^{(1)})$ and $(\sigma, V^{(2)})$ be two completely contractive covariant representations of $E_1$ and $E_2$ on $\mathcal{H},$ respectively. 
 Then the operator matrix $$ M= \begin{bmatrix}
 \widetilde{V}^{(1)}(I_{E_1} \otimes \widetilde{V}^{(2)}) & \widetilde{V}^{(1)}(I_{E_1} \otimes (I_H - \widetilde{V}^{(2)}\widetilde{V}^{(2)^*}))^{1/2} \\
 0 & 0 \\
 \end{bmatrix}$$ is a partial isometry if and only if $(\sigma,V^{(1)})$ is a partial isometric representation.
\end{theorem}

\begin{proof}
	Suppose that  $(\sigma,V^{(1)})$ is a partial isometric representation. Note that	\begin{align*} MM^*
		= \begin{bmatrix}
			\widetilde{V}^{(1)}\widetilde{V}^{(1)^*}    &     0 \\
			0      &        0 
		\end{bmatrix}.
	\end{align*} Since $\widetilde{V}^{(1)} \widetilde{V}^{(1)^*}$ is a projection, $MM^*$ is also a projection. Thus $M$ is a partial isometry.
	
	Conversely, suppose that $M$ is a partial isometry, then $M^*MM^*=M^*.$ Since
\begin{align*} M^*MM^*
		= \begin{bmatrix}
			(I_{E_1}  \otimes \widetilde{V}^{(2)^*})\widetilde{V}^{(1)^*}\widetilde{V}^{(1)}\widetilde{V}^{(1)^*}    &     0 \\
			(I_{E_1} \otimes (I_H - \widetilde{V}^{(2)}\widetilde{V}^{(2)^*}))^{1/2}\widetilde{V}^{(1)^*}\widetilde{V}^{(1)}\widetilde{V}^{(1)^*}    &        0 
		\end{bmatrix}
	\end{align*}and 
		\begin{align*} M^*
		= \begin{bmatrix}
			(I_{E_1} \otimes \widetilde{V}^{(2)^*})\widetilde{V}^{(1)^*}    &     0 \\
			(I_{E_1} \otimes (I_H - \widetilde{V}^{(2)}\widetilde{V}^{(2)^*}))^{1/2}\widetilde{V}^{(1)^*}    &        0 
		\end{bmatrix},
	\end{align*}
	we get
	\begin{align}\label{www02}
		(I_{E_1} \otimes \widetilde{V}^{(2)^*})\widetilde{V}^{(1)^*}  = 	(I_{E_1}  \otimes \widetilde{V}^{(2)^*})\widetilde{V}^{(1)^*}\widetilde{V}^{(1)}\widetilde{V}^{(1)^*}
	\end{align} and
	\begin{align}\label{www03}
		(I_{E_1} \otimes (I_H - \widetilde{V}^{(2)}\widetilde{V}^{(2)^*}))^{1/2}\widetilde{V}^{(1)^*} = (I_{E_1} \otimes (I_H - \widetilde{V}^{(2)}\widetilde{V}^{(2)^*}))^{1/2}\widetilde{V}^{(1)^*}\widetilde{V}^{(1)}\widetilde{V}^{(1)^*}.
	\end{align}
	Now pre-multiply by $(I_{E_1} \otimes \widetilde{V}^{(2)})$ in Equation (\ref{www02}) and  $(I_{{E_1} \otimes \mathcal{H}} - (I_{E_1}  \otimes \widetilde{V}^{(2)}\widetilde{V}^{(2)^*}))^{1/2}$  in Equation (\ref{www03}), we get
	$\widetilde{V}^{(1)^*} = \widetilde{V}^{(1)^*}\widetilde{V}^{(1)}\widetilde{V}^{(1)^*}.$ Therefore  $(\sigma,V^{(1)})$ is a partial isometric representation.
\end{proof} 

\section{Powers and roots of partial isometric representations}

In this section we will discuss the powers and roots of partial isometric representations. From Theorem \ref{www01}, we proved that if $(\sigma,V)$  is a partial isometric representation of $E$ on $\mathcal{H}$ such that $(I_{E^{\otimes(i-1)}}\otimes\widetilde{V}\widetilde{V}^*){N}(\widetilde{V}_{i-1}) \subseteq {N}(\widetilde{V}_{i-1})$ for every $i \le n$ , then
$\widetilde{V}_i$ is a partial isometry for every $i \le n.$

 The following result gives an equivalent condition of Theorem \ref{www01}, which is an analogue of \cite[Proposition 1]{EMSZ18}.

\begin{proposition}\label{www04}
	 Let $(\sigma,V)$ be a partial isometric representation of a $C^*$-correspondence $E$ on $\mathcal{H}.$ For $n\in \mathbb{N},$ the following conditions are equivalent:
	\begin{enumerate}
		\item $(I_{E^{\otimes(n-1)}}\otimes\widetilde{V}){N}(\widetilde{V}_{n})^{\perp} \subseteq {N}(\widetilde{V}_{n-1})^{\perp}.$
		\item $(I_{E^{\otimes(n-1)}}\otimes\widetilde{V}\widetilde{V}^*){N}(\widetilde{V}_{n-1}) \subseteq {N}(\widetilde{V}_{n-1}).$
	\end{enumerate}
\end{proposition}

\begin{proof}
	For $n\in \mathbb{N},$ we have
\begin{align*}
(I_{E^{\otimes(n-1)}}\otimes\widetilde{V}){N}(\widetilde{V}_{n})^{\perp}  &=(I_{E^{\otimes(n-1)}}\otimes\widetilde{V})\widetilde{V}_n^*\mathcal{H}=(I_{E^{\otimes(n-1)}}\otimes \widetilde{V}\widetilde{V}^*)\widetilde{V}_{n-1}^*\mathcal{H}\\&=(I_{E^{\otimes(n-1)}}\otimes\widetilde{V}\widetilde{V}^*){N}(\widetilde{V}_{n-1})^{\perp} .
\end{align*} It gives $(1)$ is equivalent to $(2).$
\end{proof}

The next result is a combination of Theorem \ref{www01} and Proposition \ref{www04}.

\begin{theorem}\label{www05}
	Let $(\sigma,V)$ be a partial isometric representation of $E$ on  $\mathcal{H}$. Then the following conditions are equivalent:
	\begin{enumerate}
		\item $\widetilde{V}_i$  is a partial isometry for every $i\in I_n.$
		\item $(I_{E^{\otimes(i-1)}}\otimes\widetilde{V}){N}(\widetilde{V}_{i})^{\perp} \subseteq {N}(\widetilde{V}_{i-1})^{\perp}$  for every $i\in I_n.$
	\end{enumerate}
	Moreover, $(\sigma,V)$ is a power partial isometric representation of $E$ on $\mathcal{H}$ if and only if $(I_{E^{\otimes i}}\otimes\widetilde{V}){N}(\widetilde{V}_{i+1})^{\perp} \subseteq {N}(\widetilde{V}_{i})^{\perp}$ for every $i\in \mathbb{N}.$
\end{theorem}
\begin{proof}
	We will prove it using the mathematical induction. For $n=1,$ it is obvious. Suppose that our assumption holds for $n=m.$ Now, we want to prove it is true for $n=m+1.$ Let $\eta\in$ ${N}(\widetilde{V}_{m+1})^{\perp}$ $\subseteq$ ${N}(I_{E^{\otimes m}}\otimes \widetilde{V})^{\perp},$ since $(I_{E^{\otimes m}}\otimes\widetilde{V}){N}(\widetilde{V}_{m+1})^{\perp} \subseteq {N}(\widetilde{V}_{m})^{\perp},$ we have  $(I_{E^{\otimes m}}\otimes\widetilde{V})\eta\in{N}(\widetilde{V}_{m})^{\perp}$. Since $(I_{E^{\otimes m}}\otimes\widetilde{V})$ and $\widetilde{V}_m$ are partial isometries, we get $$\|\widetilde{V}_{m+1}\eta\|=\|\widetilde{V}_m(I_{E^{\otimes m}} \otimes \widetilde{V})\eta\|=\|(I_{E^{\otimes m}} \otimes \widetilde{V})\eta\|=\|\eta\|.$$
	Therefore $\widetilde{V}_{m+1}$ is a partial isometry.
	
	Conversely, suppose that $\widetilde{V}_i$ is a partial isometry for $i\in I_{m+1}.$ From Theorem \ref{www01}, we obtain
	$$(I_{E^{\otimes(i-1)}}\otimes\widetilde{V}\widetilde{V}^*){N}(\widetilde{V}_{i-1}) \subseteq {N}(\widetilde{V}_{i-1})  \quad \quad for \quad  all \quad i\in I_{m+1}.$$ From Proposition \ref{www04}, $(I_{E^{\otimes(i-1)}}\otimes\widetilde{V}){N}(\widetilde{V}_{i})^{\perp} \subseteq {N}(\widetilde{V}_{i-1})^{\perp}$ for every $i\in I_{m+1}.$
\end{proof}

\begin{corollary}
	Let $(\sigma,V)$ be a partial isometric representation of $E$ on  $\mathcal{H}$ such that $\widetilde{V}_n$  is a partial isometry
	for some $n\ge 2.$ If $(I_{E^{\otimes n}}\otimes\widetilde{V}){N}(\widetilde{V}_{n+1})^{\perp} \subseteq {N}(\widetilde{V}_{n})^{\perp}$, then $\widetilde{V}_{n+1}$ is a partial
	isometry.
\end{corollary}

\begin{definition}
	Let $(\sigma,V)$ be a c.b.c. representation of ${E}$ on $\mathcal H.$ We say that $(\sigma,V)$ is {\rm regular} if $N(\widetilde{V})\subseteq {E} \otimes{ {{R}^{\infty}{({\widetilde{V}})}}}$ and its range $R(\widetilde{V})$ is closed, where ${ {{R}^{\infty}{({\widetilde{V}})}}}=\bigcap_{n \geq 0} { {{R}{(\widetilde{V}_n)}}}$ is the {\rm generalized range} of $(\sigma,V).$
\end{definition}

The next result characterizes the regular c.b.c. representations which is from \cite[Theorem 2.2]{AHS22}.
\begin{theorem}
Suppose that $(\sigma,V)$ is a c.b.c. representation of ${E}$ on $\mathcal H.$ Then, for all $ m, n \in \mathbb{N} $, the following conditions are equivalent:
	\begin{enumerate}
		\item $ N(\widetilde{V}) \subseteq (I_{E} \otimes \widetilde{V}_m)(E^{\otimes(m+1) }\otimes \mathcal{H}).$
		\item $ N({\widetilde{V}}_n) \subseteq(I_{E^{\otimes n}} \otimes \widetilde{V}_m)(E^{\otimes(n+m) }\otimes \mathcal{H}).$
		\item $ N({\widetilde{V}}_n) \subseteq (I_{E^{\otimes n}} \otimes \widetilde{V})(E^{\otimes(n+1) }\otimes \mathcal{H}).$
		\item $N(\widetilde{V}_{n}) = (I_{E^{\otimes n}} \otimes \widetilde{V}_m)N (\widetilde{V}_{m+n}).$
	\end{enumerate}
	
\end{theorem}

Let $(\sigma,V)$ be a c.b.c. representation of ${E}$ on $\mathcal H.$ A bounded operator  ${S}:\mathcal{H} \to E\otimes\mathcal{H}$ is said to be  {\it generalized inverse} of $\widetilde{V}$ if $\widetilde{V}{S}\widetilde{V}=\widetilde{V}$ and ${S}\widetilde{V}{S}={S}$.

\begin{lemma}\label{www06}
	Let $(\sigma, V)$ be a regular c.b.c. representation  of $E$ on  $\mathcal{H},$ and let $S$ be a generalized inverse  of $\widetilde{V}$. Then $(I_{E^{\otimes n}}\otimes S){N}(\widetilde{V}_n) \subseteq {N}(\widetilde{V}_{n+1})$ for all $n\in \mathbb{N}.$
\end{lemma}
\begin{proof}
Since $(\sigma, V)$ is a regular, let $\eta\in$ ${N}(\widetilde{V}_n) \subseteq {E^{\otimes n}} \otimes {R}(\widetilde{V}),$ then there exists $\zeta\in E^{\otimes n+1}\otimes\mathcal{H}$ such that  $\eta=(I_{E^{\otimes n}}\otimes \widetilde{V})\zeta.$ It follows that
	\begin{align*}
		\widetilde{V}_{n+1}(I_{E^{\otimes n}}\otimes S)\eta &=\widetilde{V}_{n}(I_{E^{\otimes n}}\otimes \widetilde{V}S\widetilde{V})\zeta=\widetilde{V}_{n}(I_{E^{\otimes n}}\otimes \widetilde{V})\zeta=\widetilde{V}_n (\eta)=0.
	\end{align*}
	Therefore $(I_{E^{\otimes n}}\otimes S){N}(\widetilde{V}_n) \subseteq {N}(\widetilde{V}_{n+1}).$
\end{proof}
Now we present the connection between partial isometric and power partial isometric representations.
\begin{corollary}
	Let $(\sigma, V)$ be a regular c.b.c. representation of $E$ on $\mathcal{H}$. Then $(\sigma, V)$ is a partial isometric representation if and only if $(\sigma, V)$ is a power partial isometric representation.
\end{corollary}
\begin{proof}
	Since $(\sigma, V)$ is a partial isometric representation, that is, $\widetilde{V}\widetilde{V}^*\widetilde{V}=\widetilde{V}$ and $\widetilde{V}^*\widetilde{V}\widetilde{V}^*=\widetilde{V}^*.$ Clearly $\widetilde{V}^*$ is a generalized inverse of $\widetilde{V}$. From Lemma \ref{www06}, $(I_{E^{\otimes n}}\otimes \widetilde{V}^*){N}(\widetilde{V}_n) \subseteq {N}(\widetilde{V}_{n+1})$ for all $n\in \mathbb{N},$ or
	equivalently, $(I_{E^{\otimes n}}\otimes\widetilde{V}){N}(\widetilde{V}_{n+1})^{\perp} \subseteq {N}(\widetilde{V}_{n})^{\perp}$ for all $n\in \mathbb{N}.$ Then from Theorem \ref{www05}, $(\sigma, V)$ is a power partial isometric representation.
\end{proof}

\begin{example}\label{www07}
	Suppose that $E$ is an $n$-dimensional Hilbert space with the orthonormal basis $\{\delta_i\}_{i\in I_{n}}$ and $\mathcal{H}$ is a Hilbert space with the orthonormal basis $\{e_m :\: m\ge 0\}.$ Let $(\rho, S^{w})$ be the unilateral weighted shift c.b.c. representation (see \cite{AHS22,Saini1}) of $E$ on $\mathcal{H}$ defined by
	$$S^{w}(\delta_i)=V_i \:\: \mbox{and}\:\:\: \rho(b)=b I_{\mathcal{H}}, \:\:b \in \mathbb{C},$$ where $V_{i}(e_{m})=w_{i,m}e_{nm+i}$ for all $m\ge 0,i\in I_n$ and $\{w_{i,m} : i \in I_n,\:\:m\ge 0 \}$ is a bounded set of nonnegative real numbers.
	Let $B\subseteq\mathbb{N}$ and $(\alpha_m)_{m\in \mathbb{N}}$ be the sequence such that
	$$\begin{cases}
		\alpha_m=0 & \text{if }  {m\in B}; \\
		\alpha_m=1 & \text{if }   {m\notin B}.
	\end{cases}$$
	The unilateral weighted shift c.b.c. is defined by $$S^{w,{B}}(\delta_i)(e_{m})=V_{i,B}(e_{m})=w_{i,m}\alpha_me_{nm+i} \quad \mbox{for all}\quad m\ge 0.$$ It is easy to verify that $N(V_{i,B})=\overline{span} \{e_m : \: m\in B\}$ and $N(V_{i,B}^k)= \bigcup_{p=1}^{k}\overline{span} \{e_{m} : \: n^{p-1}m +\sum_{l=2}^{p} n^{p-l}i \in B\}$ for every $i\in I_n$ and $k\in \mathbb{N}.$ This implies that $$V_{i,B}(N(V_{i,B}^{k+1})^{\perp})\subseteq N(V_{i,B}^{k})^{\perp}\quad \mbox{for all}\quad i\in I_n.$$
\end{example}
The following result is a combination of Example \ref{www07} and Theorem \ref{www05}:
\begin{proposition}
	\begin{enumerate}
		\item $(\rho, S^{w,{B}})$ is a partial isometric representation of $E$ on $\mathcal{H}$ if and only if  $w_{i,m}=1$ for every $m\notin B$ and $i\in I_n.$
		\item  Every partial isometric unilateral weighted shift c.b.c. representation $(\rho, S^{w,{B}})$  of $E$ on $\mathcal{H}$ is a power partial isometric representation.
	\end{enumerate}
\end{proposition}

%\begin{theorem}
%	Let $(\sigma,V)$ be a completely contractive, covariant representation of $E$ on  $\mathcal{H}$ such that ${N}(I_{E} \ot \widetilde{V})= {N}(\widetilde{V}_2).$ Let $\widetilde{V}_n$ be a partial isometry for some $n \ge 2,$ then  $(\sigma,V)$ is a partial isometric representation of $E$ on  $\mathcal{H}.$
%\end{theorem}
%\begin{proof}
%	 Since ${N}(I_{E} \ot \widetilde{V})= {N}(\widetilde{V}_2),$ we have ${N}(I_{E^{\ot (n-1)}} \ot \widetilde{V})= {N}(\widetilde{V}_n)$ for every $n \ge 2.$ Suppose that $\widetilde{V}_n$ is a partial isometry for some $n \ge 2,$ then $\widetilde{V}_n(I_{E^{\otimes n}\otimes \mathcal{H}}-\widetilde{V}_n^*\widetilde{V}_n)=0.$ This implies that $(I_{E^{\otimes n}\otimes \mathcal{H}}-\widetilde{V}_n^*\widetilde{V}_n) \in {N}(\widetilde{V}_n)={N}(I_{E^{\ot (n-1)}} \ot \widetilde{V}),$ and hence $(I_{E^{\otimes{n-1}}}\otimes \widetilde{V})(I_{E^{\otimes n}\otimes \mathcal{H}}-\widetilde{V}_n^*\widetilde{V}_n)=0.$ It follows that $(I_{E^{\otimes{n-1}}}\otimes \widetilde{V})=(I_{E^{\otimes{n-1}}}\otimes \widetilde{V})\widetilde{V}_n^*\widetilde{V}_n=(I_{E^{\otimes{n-1}}}\otimes \widetilde{V})\widetilde{V}_n^*\widetilde{V}_{n-1}(I_{E^{\otimes{n-1}}}\otimes \widetilde{V}).$  Defining $S=\widetilde{V}_n^*\widetilde{V}_{n-1}$
%	we find that $S$ is a contraction such that $(I_{E^{\otimes{n-1}}}\otimes \widetilde{V})=(I_{E^{\otimes{n-1}}}\otimes \widetilde{V})S(I_{E^{\otimes{n-1}}}\otimes \widetilde{V}).$ Then $(I_{E^{\otimes{n-1}}}\otimes \widetilde{V})$ is a partial isometry, and so $\wV$ is a partial isometry. 
%\end{proof}

The following theorem is the main result of this section which is a generalization of \cite[Theorem 5]{EMSZ18}.
\begin{theorem}
Let $(\sigma,V)$ be a completely contractive covariant representation of a full $C^*$-correspondence $E$ on $\mathcal{H}.$ Suppose that $\widetilde{V}_k$ is a partial isometry for some $k \ge 2.$ Then the following conditions are equivalent:
\begin{enumerate}
\item $(\sigma,V)$ is a partial isometric representation.
\item $(a)$ For all $\eta\in {N}(\widetilde{V}_k) \ominus {N}(I_{E^{\otimes{k-1}}}\otimes \widetilde{V})$ we have $$\|(I_{E^{\otimes{k-1}}}\otimes \widetilde{V})\eta\|=\|\eta\|.$$
$(b)$ $(I_{E^{\otimes{k-1}}}\otimes \widetilde{V})({N}(\widetilde{V}_k)^{\perp}) \perp (I_{E^{\otimes{k-1}}}\otimes \widetilde{V})({N}(\widetilde{V}_k)\ominus {N}(I_{E^{\otimes{k-1}}}\otimes \widetilde{V})).$
\end{enumerate}
\end{theorem}
\begin{proof}
	$(1)\Rightarrow (2).$ Suppose that $ \widetilde{V}$ is a partial isometry, then $(I_{E^{\otimes{k-1}}}\otimes \widetilde{V})$ is also a partial isometry. Let $\eta\in {N}(\widetilde{V}_k) \ominus {N}(I_{E^{\otimes{k-1}}}\otimes \widetilde{V}) \subseteq {N}(I_{E^{\otimes{k-1}}}\otimes \widetilde{V})^{\perp},$ then $\|(I_{E^{\otimes{k-1}}}\otimes \widetilde{V})\eta\|=\|\eta\|.$ Next we want to prove that $(I_{E^{\otimes{k-1}}}\otimes \widetilde{V})({N}(\widetilde{V}_k)^{\perp}) \perp (I_{E^{\otimes{k-1}}}\otimes \widetilde{V})({N}(\widetilde{V}_k)\ominus {N}(I_{E^{\otimes{k-1}}}\otimes \widetilde{V})).$ Let $\eta\in$ ${N}(\widetilde{V}_k)^{\perp} \subseteq {N}(I_{E^{\otimes{k-1}}}\otimes \widetilde{V})^{\perp}$ and $\zeta\in{N}(\widetilde{V}_k)\ominus {N}(I_{E^{\otimes{k-1}}}\otimes \widetilde{V}) \subseteq {N}(I_{E^{\otimes{k-1}}}\otimes \widetilde{V})^{\perp},$ then $\|(I_{E^{\otimes{k-1}}}\otimes \widetilde{V})\eta\|=\|\eta\|$ and $\|(I_{E^{\otimes{k-1}}}\otimes \widetilde{V})\zeta\|=\|\zeta\|.$ Since $(I_{E^{\otimes{k-1}}}\otimes \widetilde{V})$ is a partial isometry, we have
	\begin{align*}
		\|(I_{E^{\otimes{k-1}}}\otimes \widetilde{V})\eta+(I_{E^{\otimes{k-1}}}\otimes \widetilde{V})\zeta\|^2 &=\|(I_{E^{\otimes{k-1}}}\otimes \widetilde{V})(\eta+\zeta)\|^2=\|\eta+\zeta\|^2 =\|\eta\|^2+ \|\zeta\|^2 \\& =\|(I_{E^{\otimes{k-1}}}\otimes \widetilde{V})\eta\|^2+\|(I_{E^{\otimes{k-1}}}\otimes \widetilde{V})\zeta\|^2.
	\end{align*}
	It follows that $(I_{E^{\otimes{k-1}}}\otimes \widetilde{V})({N}(\widetilde{V}_k)^{\perp}) \perp (I_{E^{\otimes{k-1}}}\otimes \widetilde{V})({N}(\widetilde{V}_k)\ominus {N}(I_{E^{\otimes{k-1}}}\otimes \widetilde{V})).$

	$(2)\Rightarrow (1).$ Suppose $(2)$ holds. Since $\widetilde{V}_k$ is a partial isometry and $\widetilde{V}$ is a contraction. For $\eta\in {N}(\widetilde{V}_k)^{\perp}$, we have $$\|\eta\|=\|\widetilde{V}_k \eta\|=\|\widetilde{V}_{k-1}(I_{E^{\otimes{k-1}}}\otimes \widetilde{V})\eta \| \le \|(I_{E^{\otimes{k-1}}}\otimes \widetilde{V})\eta \| \le \|\eta\|.$$ This implies that $\|(I_{E^{\otimes{k-1}}}\otimes \widetilde{V})\eta \|=\|\eta\|$ for all  $\eta\in{N}(\widetilde{V}_k)^{\perp}.$ Since ${N}(I_{E^{\otimes{k-1}}}\otimes \widetilde{V})^{\perp}={N}(\widetilde{V}_k)^{\perp} \oplus ({N}(\widetilde{V}_k) \ominus {N}(I_{E^{\otimes{k-1}}}\otimes \widetilde{V})).$ For each $\eta \in {N}(I_{E^{\otimes{k-1}}}\otimes \widetilde{V})^{\perp},$ we have $\eta=\eta_1+\eta_2 \in {N}(\widetilde{V}_k)^{\perp} \oplus ({N}(\widetilde{V}_k) \ominus {N}(I_{E^{\otimes{k-1}}}\otimes \widetilde{V})).$
	From $(2)$ we get
	\begin{align*}
		&\|(I_{E^{\otimes{k-1}}}\otimes \widetilde{V})\eta\|^2 =\|(I_{E^{\otimes{k-1}}}\otimes \widetilde{V})\eta_1+(I_{E^{\otimes{k-1}}}\otimes \widetilde{V})\eta_2\|^2\\&=\|(I_{E^{\otimes{k-1}}}\otimes \widetilde{V})\eta_1\|^2+\|(I_{E^{\otimes{k-1}}}\otimes \widetilde{V})\eta_2\|^2 =\|\eta_1\|^2+\|\eta_2\|^2=\|\eta\|^2.
	\end{align*}
	Therefore $(I_{E^{\otimes{k-1}}}\otimes \widetilde{V})$ is a partial isometry. Since $E$ is full and $\sigma$ is nondegenerate, $(\sigma,V)$ is a partial isometric representation.
\end{proof}

The following result is a generalization of \cite[Theorem A]{BCG80}.
\begin{remark}
	Let $(\sigma,V)$ be a completely contractive covariant representation of a full $C^*$-correspondence $E$ on $\mathcal{H}.$ Suppose that $\widetilde{V}_k$ is a partial isometry for some $k \ge 2.$ If ${N}(I_{E} \otimes \widetilde{V})= {N}(\widetilde{V}_2),$ then $(\sigma,V)$ is a partial isometric representation.
\end{remark}

\begin{remark}
		Let $(\sigma,V)$ be a completely contractive covariant representation of a full $C^*$-correspondence $E$ on $\mathcal{H}.$ Suppose that $\widetilde{V}_k$ is a partial isometry for some $k \ge 2$ such that \begin{itemize}
 	\item[(1)] $\|(I_{E^{\otimes{k-1}}}\otimes \widetilde{V})\eta\|=\|\eta\|$ for all $\eta\in{N}(\widetilde{V}_k) \ominus {N}(I_{E^{\otimes{k-1}}}\otimes \widetilde{V}).$
 	\item[(2)] $(I_{E^{\otimes(k-1)}}\otimes\widetilde{V}){N}(\widetilde{V}_{k})^{\perp} \subseteq {N}(\widetilde{V}_{k-1})^{\perp}.$
 \end{itemize}
	Then $(\sigma,V)$ is a partial isometric representation. Indeed, let $\eta\in {N}(\widetilde{V}_k)\ominus {N}(I_{E^{\otimes{k-1}}}\otimes \widetilde{V}),$ then $(I_{E^{\otimes(k-1)}}\otimes\widetilde{V})\eta \in {N}(\widetilde{V}_{k-1}).$ Since $(I_{E^{\otimes(k-1)}}\otimes\widetilde{V}){N}(\widetilde{V}_{k})^{\perp} \subseteq {N}(\widetilde{V}_{k-1})^{\perp},$ we obtain  $(I_{E^{\otimes{k-1}}}\otimes \widetilde{V})({N}(\widetilde{V}_k)^{\perp}) \perp (I_{E^{\otimes{k-1}}}\otimes \widetilde{V})({N}(\widetilde{V}_k)\ominus {N}(I_{E^{\otimes{k-1}}}\otimes \widetilde{V})).$
\end{remark}

\subsection{Wold-type decomposition for regular partial isometric representations}

Ezzahraoui, Mbekhta, and Zerouali \cite{EMZ15} discussed a Wold-type decomposition for regular operators using growth condition. Rohilla, Veerabathiran, and Trivedi \cite{AHS22} generalized this decomposition for regular covariant representations with growth condition. In \cite{EMZ21}, Ezzahraoui, Mbekhta, and Zerouali proved a Wold-type decomposition for regular operators with regular Moore-Penrose inverse. Saini \cite{Saini1} generalized this decomposition for {regular} covariant representations whose Moore-Penrose inverse is regular.

Now we recall the following notion of bi-regular covariant representation and it's Wold-type decomposition from \cite{Saini1}.

\begin{definition}
	Let $(\sigma,V)$ be a regular c.b.c. representation of ${E}$ on $\mathcal H.$ We say that $(\sigma,V)$ is {\rm bi-regular} if $N(I_{E^{\otimes n}}\otimes \widetilde{V}^{\dagger})\subseteq R(\widetilde{V}^{\dagger {(n)}})$ for all $n\in \mathbb{N},$ that is, $\widetilde{V}^{\dagger}$ is regular.
\end{definition}
Note that if $(\sigma,V)$ is a regular c.b.c. representation of ${E}$ on $\mathcal H.$ Then $\widetilde{V}^*$ is also regular, that is, $N(I_{E^{\otimes n}}\otimes \widetilde{V}^{*})\subseteq R(\widetilde{V}_n^{*})$ for all $n\in \mathbb{N}.$ Let $(\sigma,V)$ be a regular partial isometric representation of ${E}$ on $\mathcal H,$ then $\widetilde{V}^*=\widetilde{V}^{\dagger},$ and hence $(\sigma,V)$ is bi-regular.
\begin{definition}
 Let $(\sigma, V)$ be a c.b.c. representation of $E$ on $\mathcal{H}$ with closed range. The Cauchy dual of $(\sigma, V)$ \cite{Saini1} is defined by $\widetilde{V}': E \otimes \mathcal{H} \to \mathcal{H}$ by  $$\widetilde{V}':=\widetilde{V}(\widetilde{V}^*\widetilde{V})^\dagger.$$ 	
\end{definition}
Suppose that $(\sigma,V)$ is a partial isometric representation of $E$ on $\mathcal{H}.$ Then by \cite[Proposition 3.3]{Saini1}, $\widetilde{V}'=\widetilde{V}.$

\begin{proposition}
	Let $(\sigma,V)$ be a bi-regular representation of $E$ on $\mathcal H.$ Then $$\mathcal{H} = [\mathcal{H}\ominus \widetilde{V}(E\otimes \mathcal{H})]_{\widetilde{V}} \oplus R^{\infty}({\widetilde{V}'})=[\mathcal{H}\ominus \widetilde{V}(E\otimes \mathcal{H})]_{\widetilde{V}'} \oplus R^{\infty}({\widetilde{V}}).$$
\end{proposition}

\begin{corollary}
	Let $(\sigma,V)$ be a regular partial isometric representation of ${E}$ on $\mathcal H.$ Then $$\mathcal{H} = [\mathcal{H}\ominus \widetilde{V}(E\otimes \mathcal{H})]_{\widetilde{V}} \oplus R^{\infty}({\widetilde{V}}).$$ 
\end{corollary}

\subsection*{Acknowledgement}
The authors thank the reviewer for carefully reading the manuscript and the editor for suggesting changes in the manuscript. Dimple Saini is supported by UGC fellowship (File No:16-6(DEC. 2018)/2019(NET/CSIR)). Harsh Trivedi is supported by MATRICS-SERB  
Research Grant, File No: MTR/2021/000286, by 
SERB, Department of Science \& Technology (DST), Government of India. Shankar Veerabathiran thanks IMSc Chennai for postdoc fellowship. Saini and Trivedi acknowledge the DST-FIST program (Govt. of India) FIST - No. SR/FST/MS-I/2018/24.

\end{document}